\documentclass{amsart}

\usepackage[latin1]{inputenc}
\usepackage{amssymb,amsthm}
\usepackage[all]{xy}
\usepackage{amsmath}
\usepackage{indentfirst}
\usepackage{fancyhdr}
\usepackage{amsfonts}
\usepackage{textcomp}
\usepackage{graphicx}
\usepackage{enumerate}
\usepackage{verbatim}
\usepackage{hyperref}
\usepackage{amsthm}
\usepackage{eucal}
\usepackage{tikz}
\usepackage{amssymb}
\usepackage{mathrsfs}

\usepackage[left, pagewise, running]{lineno}

\newtheorem{de}{Definition}[section]

\newtheorem{pro}[de]{Proposition}
\newtheorem{nota}[de]{Note}

\newtheorem{es}[de]{Example}
\newtheorem{claim}[de]{Claim}
\newtheorem{lem}[de]{Lemma}
\newtheorem{co}[de]{Corollary}
\newtheorem{re}[de]{Remark}

\title{On the ampleness and bigness of non-integral divisors}
\author{Stefano Urbinati}
\date{}

\begin{document}

\address{Universit`a degli Studi di Padova\\ Mathematics department \\ Via Trieste 63 - 35121  Padova - Italy}
\email{urbinati.st@gmail.com}

\begin{abstract} Given a Weil non-integral divisor $D$,  it is natural to associate it the line bundle of its integral part $\mathcal{O}_X([D])$. 
In this work we study which of the classical characterizations of ample and big divisors can be extended to non-integral divisors via the corresponding line bundles.
\end{abstract}

\maketitle

A powerful method for studying algebraic varieties is to embed them, if possible, into some projective space $\mathbb{P}^n$. 
As it is well known (\cite{MR0463157}) morphisms in projective spaces are given by line bundles and chosen sections. 
In this note we study ampleness and bigness for $\mathbb{Q}$-divisors and $\mathbb{R}$-divisors through the notion of integral part.

For positivity questions, it is very useful to discuss small perturbations of given divisors and the natural way to do so is by the formalism of $\mathbb{Q}$ and $\mathbb{R}$-divisors. 
Given an algebraic variety $X$, a Cartier \emph{$\mathbb{R}$-divisor} on $X$ is an element of the $\mathbb{R}$-vector space
$$\text{Div}_{\mathbb{R}}(X) \stackrel{\text{def}}{=} \text{Div}(X) \otimes_{\mathbb{Z}}\mathbb{R}.$$  

Equivalently $$D \in \text{Div}_{\mathbb{R}}(X) \Leftrightarrow D=\sum c_i D_i \;{\rm with} \; c_i \in \mathbb{R}, \; D_i \in \text{Div}(X).$$

The study of those classes of divisors began in the first part of the 80's. They are fundamental in the birational study of algebraic varieties, in particular for vanishing theorems (as the vanishing theorem of Kawamata and Viehweg \cite{MR2095472}). Also note that there exist singular varieties for which the canonical divisor is a $\mathbb{Q}$-divisor. 

Our aim is to give a characterization of ampleness for these two classes of divisors following the one for $\mathbb{Z}$-divisors.
%Some properties of Proposition \ref{pro:intro} (6-11) dipend only upon the numerical class of $D$ and it is well-known that they characterize ampleness \cite{MR2095471}.\\
%On the other hand, 
The main issue we would like to discuss is that some properties of integral divisors that characterize ampleness are connected to the associated line bundle, that is only well defined for integral divisors. There are different possible ways to overcome this problem. In this work we have chosen to substitute any real divisor by its integral part any time we had to consider the associated line bundle. 

Given an $\mathbb{R}$-divisor $D= \sum_i a_iD_i$ $a_i \in \mathbb{R}$, $D_i \in \text{Div}(X)$ prime divisors, we define its integral part as 
$$[D]= \sum_i [a_i]D_i \in \text{Div}(X).$$

In the first section we collect all the well known characterization of ampleness for integral divisors. In the following two sections we prove which of these  properties can be extended to $\mathbb{Q}$-divisors and $\mathbb{R}$-divisors using the notion of integral part.

In the last section we use the results of the previous sections to characterize big non-integral divisors.

\subsection*{Acknowledgments} This work is an extract of my master thesis back in 2005. I would like to thank my advisor at the time, Angelo Felice Lopez, for his great support and for having introduced me to the world of birational geometry. The author is grateful to the referee for the several suggestions.

\section{Ampleness for $\mathbb{Z}$-divisors}
All basic properties and definitions can be found in the book of Lazarsfeld \cite{MR2095471}.

We aim to give a global characterization of ampleness for $\mathbb{Z}$, $\mathbb{Q}$ and $\mathbb{R}$-divisors. We will try to understand the differences among those three classes and study the common properties. We begin recalling all the characterizations for $\mathbb{Z}$-divisors.

\begin{pro}[Ampleness for $\mathbb{Z}$-divisors]\label{pro:ai} \cite[1.2.6, 1.2.32, 1.2.23, 1.4.13, 1.4.11, 1.4.29, 1.4.23]{MR2095471}
Let $D\in \text{Div}(X)$ be an integral Cartier divisor on a normal projective variety $X$, and let $\mathcal{O}_X(D)$ be the associated line bundle (sometimes we will think $D$ as a Weil divisor by the canonical correspondence). The following statements are equivalent:

\begin{enumerate}

\item There exists a positive integer $m$ such that $\mathcal{O}_X(mD)$ is very ample;

\item Given any coherent sheaf $\mathscr{F}$ on $X$, there exists a positive integer $m_1=m_1(\mathscr{F})$ having the property that $$H^i(X, \mathscr{F}\otimes \mathcal{O}_X(mD))=0 \quad \forall i >0, \; m \geq m_1;$$

\item Given any coherent sheaf $\mathscr{F}$ on $X$, there exists a positive integer $m_2=m_2(\mathscr{F})$ such that $\mathscr{F}\otimes \mathcal{O}_X(mD)$ is globally generated $\forall m \geq m_2$;

\item There is a positive integer $m_3$ such that $\mathcal{O}_X(mD)$ is very ample $\forall m \geq m_3$;

\item For every subvariety $V \subseteq X$ of positive dimension, there is a positive integer $m=m(V)$, together with a non-zero section $0\ne s =s_V \in H^0(V, \mathcal{O}_V(mD))$, such that $s$ vanishes at some point of $V$;

\item For every subvariety $V \subseteq X$ of positive dimension, \\$\chi(V, \mathcal{O}_V(mD))\to +\infty$ as $m \to +\infty$; 

\item \textbf{(Nakai-Moishezon-Kleiman criterion)} $$\int_V c_1(\mathcal{O}_X(D))^{\dim(V)}>0$$ for every positive-dimensional irreducible subvariety $V \subseteq X$;

\item \textbf{(Seshadri's criterion)} There exists a real number $\varepsilon > 0$ such that $$\frac{(D.C)}{\text{mult}_xC}\geq \varepsilon$$ for every point $x \in X$ and every irreducible curve $C \subseteq X$ passing through $x$;

\item Let $H$ be an ample divisor. There exists a positive number $\varepsilon >0$ such that $$\frac{(D.C)}{(H.C)}\geq \varepsilon$$ for every irreducible curve $C\subseteq X$;

\item \textbf{(Via cones)} $\overline{\text{NE}}(X)-\{0\} \subseteq D_{>0}$. 

\item There exists a neighborhood $U$ of $[D]_{num} \in  \text{N}^1(X)_{\mathbb{R}}$ such that \\
$U\backslash \{[D]_{num}\} \subseteq \text{Amp}(X)$.

\end{enumerate}

\end{pro}

\begin{pro}\label{pro:ampio} Every divisor that verifies at least one of the properties of the Proposition \ref{pro:ai} is an ample $\mathbb{Z}$-divisor.
\end{pro}

\section{Ampleness for $\mathbb{Q}$-divisors and $\mathbb{R}$-divisors}

The simpler way to extend the definition of ampleness to non-integral divisors is the following.

\begin{de}[Amplitude for $\mathbb{Q}$ and $\mathbb{R}$-divisors] \label{de:air} A $\mathbb{Q}$-divisor \\ $D \in \text{Div}_{\mathbb{Q}}(X)$ (resp. $\mathbb{R}$-divisor $D \in \text{Div}_{\mathbb{R}}(X)$) is said to be \emph{ample} if it can be written as a finite sum $$D= \sum c_iA_i$$ where $c_i >0$ is a positive rational (resp. real) number and $A_i$ is an ample Cartier divisor.

\end{de}

Our aim is to understand when this definition is equivalent to some of the properties of Proposition \ref{pro:ai}. Some of them are automatically transferable (and for those we will directly prove that they are equivalent to the concept of ampleness for $\mathbb{Q}$ and $\mathbb{R}$-divisors). Other properties depend directly on the line bundle and not on the divisor (as \ref{pro:ai} $(1)$ or $(2)$). In this cases we chose to substitute the divisor $mD$ by its integral part $[mD]$ whenever the divisor was not integral and see if the equivalence was still valid.

We begin with a negative answer for one of the statements.

\begin{es}  \label{es:nna}  The affirmation (1) in Proposition \ref{pro:ai}, when we replace $mD$ with its integral part, is not equivalent to the concept of ampleness for $\mathbb{Q}$-divisors.

To prove that it is not possible to extend (1), or rather that the property of existence of an integer $m$ such that $[mD]$ is very ample is not sufficient to characterize the amplitude for a $\mathbb{Q}$-divisor,  it is enough to find an example. We consider a ruled rational surface $X_e$, $e \geq 2$ defined as $\mathbb{P}(\mathscr{E})= \mathbb{P}(\mathcal{O}\oplus\mathcal{O}(-e))$ over $\mathbb{P}^1$.
We now consider a divisor in the form $$D= \frac{3}{2}C_0 + (e+1)f$$ where $C_0$ is a section and $f$ is a fiber of the canonical morphism over $\mathbb{P}^1$. \\
Now $$[D]=C_0 + (e+1)f$$ is a very ample divisor by \cite{MR0463157} (Theorem V.2.17) but 
$$D.C_0= \left( \frac{3}{2}C_0 + (e+1)f \right).C_0  = 1 - \frac{e}{2} \leq0$$ so that $D$ is not ample.
\end{es}

\begin{re}
For the same reason the affirmation (5) in Proposition \ref{pro:ai}, when we replace $mD$ with its integral part, is not equivalent to the concept of ampleness for $\mathbb{Q}$-divisors. In this case we will replace ``$\exists m$'' by ``$ \forall \;m \geq m_4$''.
\end{re}

\begin{nota}[A useful way to write divisors] \label{nota:utile}
Let $D$ be an $\mathbb{R}$-divisor, and suppose that $D=\sum a_iD_i$ where $a_i \in \mathbb{R}$ and $D_i\in \text{Div}(X)$, \emph{not necessarily prime}. For every integer $m \geq 1$ we can write 
$$mD= m \sum a_i D_i= \sum \left([ma_i]D_i + \{ma_i\}D_i \right),$$
where $[\cdot]$ is the integral part and $\{\cdot\}$ is the fractional part, so that we obtain:
$$[mD]= \left[ \sum \left(  [m a_i]D_i +  \{m a_i\}D_i\right)\right]= \sum [m a_i]D_i + \left[ \sum\{m a_i\}D_i\right].$$ 
Now $\left\{ \left[\sum \{m a_i\}D_i\right]\right\}=\{T_m\}$ is a finite set of integral divisors, $\{T_m\}=\{T_{k_1}, \dots, T_{k_s}\}$.
\end{nota}

\begin{re} If $D$ is an integral divisor, $D$ is ample in the sense of $\mathbb{Z}$-divisors (\ref{pro:ai}) if and only if it is ample in the sense of $\mathbb{R}$-divisors (\ref{de:air}).
\end{re}
\begin{proof}
If $D$ is ample in the sense of $\mathbb{Z}$-divisors, obviously $D$ can be written as $1 \cdot D$ where $D$ is an ample divisor, so that it is an ample real divisor.\\
If $D=\sum c_i A_i $ in an ample $\mathbb{R}$-divisor, by Note \ref{nota:utile} we can write $[mD]=\sum[ma_i]A_i +T_k$ for finitely many divisors $T_k$. As $A_1$ is ample, by Proposition \ref{pro:ai}, there exists an integer $r>0$ such that $rA_1 + T_k$ is globally generated for every $k$ and there exists an integer $s>0$ such that $tA_i$ is very ample for all $i$ and for all $t\geq s$. Then, if $m \geq \frac{r + s}{a_i}\; \forall i$, we have $$[mD]= \sum_{i\geq 2}[ma_i]A_i + ([ma_1] -r)A_1 + (rA_1 + T_k)$$ that is a sum of a very ample and a globally generated integral divisor, that is very ample. But in this case $[mD]=mD$ and we get the statement.
\end{proof}

We can now state the following proposition: 

\begin{pro}[Ampleness for $\mathbb{Q}$-divisors]\label{pro:ar}
Let $D\in \text{Div}_{\mathbb{Q}}(X)$ be a Cartier divisor on a normal projective variety $X$, and let $\mathcal{O}_X([D])$ be the associated line bundle (sometimes we will think $[D]$ as a Weil divisor by the canonical correspondence). The following statements are equivalent to the definition of ampleness for $\mathbb{Q}$-divisors (Definition $\ref{de:air}$):

\begin{itemize}

\item[(I)] Given any coherent sheaf $\mathscr{F}$ on $X$, there exists a positive integer $m_1=m_1(\mathscr{F})$ having the property that $$H^i(X, \mathscr{F}\otimes \mathcal{O}_X([mD]))=0 \quad \forall i >0, \; m \geq m_1;$$

\item[(II)] Given any coherent sheaf $\mathscr{F}$ on $X$, there exists a positive integer $m_2=m_2(\mathscr{F})$ such that $\mathscr{F}\otimes \mathcal{O}_X([mD])$ is globally generated $\forall m \geq m_2$;

\item[(III)] There is a positive integer $m_3$ such that $\mathcal{O}_X([mD])$ is very ample $\forall m \geq m_3$;

\item[(IV)] For every subvariety $V \subseteq X$ of positive dimension, there is a positive integer $m_4=m_4(V)$, such that for every $m \geq m_4$ there exists a non-zero section $0\ne s =s_{V,m} \in H^0(V, \mathcal{O}_V([mD]))$, such that $s$ vanishes at some point of $V$;

\item[(V)] For every subvariety $V \subseteq X$ of positive dimension, \\$\chi(V, \mathcal{O}_V([mD]))\to +\infty$ as $m \to +\infty$; 

\item[(VI)] \textbf{(Nakai-Moishezon-Kleiman criterion)} $$\int_V c_1(\mathcal{O}_X(D))^{\dim(V)}>0$$ for every irreducible positive-dimensional subvariety $V \subseteq X$;

\item[(VII)] \textbf{(Seshadri's criterion)} There exists a real number $\varepsilon > 0$ such that $$\frac{(D.C)}{\text{mult}_xC}\geq \varepsilon$$ for every point $x \in X$ and every irreducible curve $C \subseteq X$ passing through $x$;

\item[(VIII)] Let $H$ be an ample divisor. There exists a positive number $\varepsilon >0$ such that $$\frac{(D.C)}{(H.C)}\geq \varepsilon$$ for every irreducible curve $C\subseteq X$;

\item[(IX)] \textbf{(Via cones)} $\overline{\text{NE}}(X)-\{0\} \subseteq D_{>0}$. 

\item[(X)] There exists a neighborhood $U$ of $[D]_{num} \in  \text{N}^1(X)_{\mathbb{R}}$ such that \\
$U\backslash \{[D]_{num}\} \subseteq \text{Amp}(X)$.

\end{itemize}

\end{pro}

\begin{proof}
\begin{claim} Either one of (I), (II), (III), (IV) and (V) implies ampleness for $\mathbb{Q}$-divisors.
\end{claim}
\begin{proof} \emph{}

\begin{itemize}
\item[(I) $\Rightarrow$ Ample] Let $a \in \mathbb{N}$ such that $aD \in \text{Div}(X)$ and let $m_0=\lceil \frac{m_1}{a} \rceil \geq 1$; if $m \geq m_0 \Rightarrow am \geq am_0 = a \lceil \frac{m_1}{a} \rceil \geq a\cdot \frac{m_1}{a}=m_1$.\\
Then by hypothesis $$H^i(\mathscr{F}([amD] )=0 \qquad \forall i>0,$$
but $H^i(\mathscr{F}([amD] )=H^i(\mathscr{F}(m (aD) )=0$ so that by Proposition \ref{pro:ampio} $aD$ is an ample integral divisor and so $D=\frac{1}{a}(aD)$ is an ample $\mathbb{Q}$-divisor.

\item[(*) $\Rightarrow$ Ample] The implications $(*) \Rightarrow$ Ample, where $(*)$ is one of $(II), (III), (IV), (V)$, can be proved in a same way.
\end{itemize}
\end{proof}

Now we want to show that if a property holds for every multiple of $mD$ than it is also valid for $D$; we will use this simple fact:

\begin{lem}\label{lem:dr}
Let $D \in \text{Div}_{\mathbb{Q}}(X)$ and let $k \in \mathbb{N}$ such that $kD \in \text{Div}(X)$. Then for every $m\in \mathbb{N}$ there exist $i, t \in \mathbb{N}$ such that $$[mD]= tkD + [iD], \quad 0\leq i\leq k-1.$$ 
\end{lem}

\begin{proof}
$D$ is a finite sum $D= \sum a_j D_j$ where $D_j \in \text{Div}(X)$ are prime divisors and $a_j \in \mathbb{Q}$. Now 
$$[mD]= \sum[ma_j]D_j;$$ and we can always write $m= tk + i$ with $0\leq i \leq k-1$ so that $ma_j= tka_j + ia_j$ where $ka_j$ is an integer and therefore $[ma_j]= tka_j +[ia_j]$.
Hence: 
\begin{align*}
[mD]= \sum_j[ma_j]D_j&= \sum_j (tka_j + [ia_j])D_j = \\
 &=tk\sum_ja_jD_j + \sum_j[ia_j]D_j=tkD + [iD].
\end{align*}
\end{proof}

\begin{claim}
Ample implies either (I), (II), (III) and (IV).
\end{claim}

\begin{proof}
Let us consider $k \in \mathbb{N}$ such that $kD=H$ is an ample integral divisor and let us use the notation of Lemma \ref{lem:dr}.

\begin{itemize}
\item[Ample $\Rightarrow$ (I)] Consider $\mathscr{F}_i= \mathscr{F}([iD])$ for $0\leq i \leq k-1$ and $n_i=n_i(\mathscr{F}_i)$ such that $H^j(\mathscr{F}([iD])\otimes \mathcal{O}_X(n k D)) =0$ for every $j>0$ and every $n \geq n_i$. \\ Then the assertion holds with $m_1= k(\max_i n_i) $.

\item[Ample $\Rightarrow$ (II)]If $H$ is ample, for every coherent sheaf $\mathscr{F}$ there exists an integer $m_0 = m_0(\mathscr{F})$ such that $\mathscr{F}(mH)$ is globally generated for every $m\geq m_0$. Consider $\mathscr{F}_i= \mathscr{F}([iD])$ for $0\leq i \leq k-1$ and $m_i=m_i(\mathscr{F}_i)$ such that $\mathscr{F}_i(mkD)$ is globally generated for every $m \geq m_i$. \\ Then the assertion holds with $m_2= k(\max_i m_i) $.

\item[Ample $\Rightarrow$ (III)] If $H$ is ample, by Proposition \ref{pro:ampio}, for every $i$ with $0\leq i\leq k-1$, there exists an integer $t_i$ such that $tH + [iD]$ is globally generated for every $t\geq t_i$. Also there exists $s \in \mathbb{N}$ such that $tH$ is very ample for every $t \geq s$.\\
Let $r = \max t_i$ and $t \geq s + r$. We get 
$$[mD]= tH + [iD]= (t-r)H +(rH + [iD])$$ that is a  very ample divisor because it is a sum of an ample and a globally generated divisor.
To conclude we get the statement for $m_3 = k(s+r)$.
 
\item[Ample $\Rightarrow$ (IV)] For what we said above \emph{Ample $\Rightarrow$ (III)} and obviously \emph{(III)$\Rightarrow$ (IV)} with $m_4=m_3$.

\item[Ample $\Rightarrow$ (V)] By Asymptotic Riemann-Roch and by Lemma \ref{lem:dr} as $\mathcal{O}_X([iD])$ has rank $1$ we have that:  $$\chi(\mathcal{O}_X([mD])=\chi(\mathcal{O}_X([iD])(tH))= \frac{(H^n)}{n!}t^n + O(t^{n-1}) \stackrel{^{t \to +\infty}}{\longrightarrow} +\infty$$ by Proposition \ref{pro:ampio} because of the ampleness of $H$.
\end{itemize} 
\end{proof}

\begin{claim} \label{claim:6-10} For any $D \in \text{Div}_{\mathbb{R}}(X)$ we have that (VI), (VII), (VIII), (IX) and (X) are equivalent to the definition of ampleness for $\mathbb{R}$-divisors.
\end{claim}

\begin{itemize}

\item[(VI)] The implication \emph{Ample $\Rightarrow$ (VI)} is obvious both for $\mathbb{Q}$ and $\mathbb{R}$-divisors. In fact if $D$ is an ample $\mathbb{R}$-divisor, it is a finite sum in the form $\sum c_i A_i$ where $c_i > 0$ and $A_i$ is an ample and integral divisor so that by Nakai-Moishezon, $(D^{(\dim V)}.V)\geq (\sum c_i)^{(\dim V)} >0$.\\
The other implication is natural for $\mathbb{Q}$-divisors: in fact if we consider $D \in \text{Div}_{\mathbb{Q}}(X)$ such that (VI) holds, we also know that there exists an integer $m >0$ such that $mD$ is an integral divisor and (VI) is also valid for $mD$, so $mD$ is ample and so is $D$.\\
For $\mathbb{R}$-divisors it has been proved by Campana-Peternell \cite{MR1048532} as we will see in the next section.

\item[(VII)] The implication \emph{Ample $\Rightarrow$ (VII)} is obvious. In fact by Definition \ref{de:air} if $D$ is an ample divisor, $D=\sum c_iA_i$ where $A_i$ is an ample integral divisor and $\mathbb{R} \ni c_i \geq 0$. Then by Seshadri's criterion (Proposition \ref{pro:ai}) over $\mathbb{Z}$ for all $i$, there exists $\varepsilon_i>0$ such that $$\frac{(A_i.C)}{\text{mult}_xC}\geq \varepsilon_i$$
then
$$\frac{(D.C)}{\text{mult}_xC} \geq \sum_i(c_i \varepsilon_i)=\varepsilon>0.$$
For the other implication we know that there exists $\varepsilon > 0$ such that $\frac{(D.C)}{\text{mult}_xC}\geq \varepsilon$ for every irreducible curve $C \subseteq X$ passing through $x$. We will proceed by induction over $n= \dim X$. In $n=1$ there is nothing to prove. For every subvariety $V \subseteq X$ such that $0 < \dim V < \dim X$, $D|_V$ is ample by induction so that by Nakai-Moishezon we only need to prove that $(D^n)>0$.\\
To this end, fix any smooth point $x \in X$, and consider the blowing up in this point with exceptional divisor $E$:

$$\begin{matrix}\mu: & X' &\to &X \\
 & \cup & & \cup\\
 & E & \to & x
\end{matrix}$$
we claim that the divisor $\mu^*D-\varepsilon E$ is nef on $X'$. Then by Kleiman's Theorem
$$(D^n)_X -\varepsilon^n=\left( (\mu^*D- \varepsilon E)^n \right)_{X'}\geq 0; $$
therefore $(D^n)>0$ as required.\\
For the nefness of $(\mu^*D - \varepsilon E)$, fix an irreducible curve $C' \subset X'$ not contained in $E$ and set $C=\mu(C')$, so that $C'$ is the proper transform of $C$. Then by Lemma \ref{lem:bo} below $$(C'.E)=\text{mult}_xC.$$
On the other hand, $$(C'.\mu^*D)_{X'}=(C.D)_X$$ by the projection formula. So the hypothesis of the criterion implies that $((\mu^*D-\varepsilon E).C')\geq 0$. Since $\mathcal{O}(E)$ is a negative line bundle on the projective space $E$ the same inequality certainly holds if $C' \subset E$. Therefore $\mu^*- \varepsilon E$ is nef and the proof is complete.

\begin{lem} \label{lem:bo} (\cite{MR2095471} Lemma 5.1.10) Let $V$ be a variety and $x \in V$ a fixed point. Denote by $\mu: V' \to V$ the blowing-up of $V$ at $x$, with exceptional divisor $E \subseteq V'$. Then $$(-1)^{(1 + \dim V)}\cdot (E^{\dim V})=\text{mult}_x V.$$

\end{lem}

\item[(VIII)] The inequality is equivalent to the condition that $D- \varepsilon H$ be nef. If we consider that the inequality holds, then we have that $D= (D - \varepsilon H) + \varepsilon H$ is ample. \\
Conversely, by the openness of the ample cone, if $D$ is ample then $D - \varepsilon H$ is even ample for $0 \leq \varepsilon \ll 1$.

\item[(IX) \& (X)] The proof of this equivalence is the same as the one we used in Proposition \ref{pro:ai} (\emph{Ample $\Rightarrow$ (VIII) $\Rightarrow$ (IX) $\Rightarrow$ (X) $\Rightarrow$ Ample}).

\end{itemize}

\end{proof}

\section{Ampleness for $\mathbb{R}$-divisors}
In this section we will study the case of $\mathbb{R}$-divisors. We will see that in this case we are only able to prove a partial correspondence to Proposition \ref{pro:ai}.

\vskip 0.5cm

We are now beginning a discussion similar to that one made for $\mathbb{Q}$-divisors:

\begin{claim} The affirmations (1) and (4) in Proposition \ref{pro:ai}, when we replace $mD$ with its integral part, are not equivalent to the concept of ampleness for $\mathbb{R}$-divisors.

\begin{proof}
Obviously, the example in the previous section (\ref{es:nna}) is still valid for $\mathbb{R}$-divisors.
\end{proof}
\end{claim}

\begin{claim}\label{claim:3nef} If (III) of Proposition \ref{pro:ar} holds for an $\mathbb{R}$-divisor $D$, then $D$ is nef.

\begin{proof}
Suppose that there exists an irreducible curve $C$ such that $D.C < 0 $. Since in $N^1(X)_{\mathbb{R}}$ we have that $\lim_{n \to \infty} \left[ \frac{[mD]}{m} \right]_{num}=[D]_{num}$ we get $$0 > D.C = \left(\lim_{m\to \infty}\frac{[mD]}{m}\right).C= \lim_{m\to \infty}\frac{[mD].C}{m} \geq 0$$ contradiction.
\end{proof}
\end{claim}

\begin{re} If (III) of Proposition \ref{pro:ar} holds for an $\mathbb{R}$-divisor $D$ over a surface, then $D$ is ample.
\label{re:surf}

\begin{proof}
We first prove that $D$ is strictly nef, i.e. $D.C>0$ for every irreducible curve $C$.\\ 
By Claim \ref{claim:3nef} $D$ is nef. Suppose that there exists an irreducible curve $C \subset X$ such that $D.C=0$.\\
We even know that for every $m>m_0$, $[mD].C \geq 1$. \\
Also $$0=mD.C = [mD].C + \{mD \}.C \Rightarrow \{mD\}.C = -[mD].C \leq -1.$$ 
If we write $D= \sum a_iD_i, \; a_i\in \mathbb{R}$, $D_i$ prime divisors, then there exists $j$ such that $(D_j.C)<0$ and so $D_j= C$, $(C^2)<0$ and $(D_i.C)\geq 0$ for all $i \ne j$.\\
By Weyl's principle (\cite{MR55555}) for every $a \in \mathbb{R}$ and for every $0< \varepsilon <1$ there exists an integer $k \gg 0$ such that $\{ka\}< \varepsilon$. Choose $\varepsilon = \frac{-1}{C^2}$ and $k>m_0$ to obtain that $\{ka_j\}(D_j.C)= \{ka_j\}(C^2) > -1$ and we get an absurd.

For the ampleness we have that $D^2>0$ because, for $m \geq m_0$, $mD=[mD] +\{mD\}$ is a sum of an ample and an effective divisor that is a big divisor (Proposition \ref{pro:bigir}).
\end{proof}
\end{re}

\begin{pro}[Ampleness for $\mathbb{R}$-divisors]\label{pro:air}
Let $D\in \text{Div}_{\mathbb{R}}(X)$ be a Cartier divisor on a normal projective variety $X$, and let $\mathcal{O}_X([D])$ be the associated line bundle (sometimes we will think $[D]$ as a Weil divisor by the canonical correspondence). The following statements are equivalent to the definition of ampleness for $\mathbb{R}$-divisors:

\begin{enumerate}

\item[i)] Given any coherent sheaf $\mathscr{F}$ on $X$, there exists a positive integer $m_2=m_2(\mathscr{F})$ such that $\mathscr{F}\otimes \mathcal{O}_X([mD])$ it is globally generated $\forall m \geq m_2$;

\item[ii)] \textbf{(Nakai-Moishezon-Kleiman criterion)} $$\int_V c_1(\mathcal{O}_X(D))^{\dim(V)}>0$$ for every positive-dimensional subvariety $V \subseteq X$;

\item[iii)] \textbf{(Seshadri's criterion)} There exists a real number $\varepsilon > 0$ such that $$\frac{(D.C)}{\text{mult}_xC}\geq \varepsilon$$ 

for every point $x \in X$ and every irreducible curve $C \subseteq X$ passing through $x$;

\item[iv)] Let $H$ be an ample divisor. There exists a positive number $\varepsilon >0$ such that $$\frac{(D.C)}{(H.C)}\geq \varepsilon$$ for every irreducible curve $C\subseteq X$;

\item[v)] \textbf{(Via cones)} $\overline{\text{NE}}(X)-\{0\} \subseteq D_{>0}$. 

\item[vi)] There exists a neighborhood $U$ of $[D]_{num} \in  \text{N}^1(X)_{\mathbb{R}}$ such that \\
$U\backslash \{[D]_{num}\} \subseteq \text{Amp}(X)$.

\end{enumerate}
\end{pro}

\begin{proof}
If $D$ is ample then $D=\sum a_i A_i$ where $A_i$ is an ample integral divisor and $a_i>0$, $a_i\in \mathbb{R}$. We can now consider $n \gg0$ such that $nA_i=H_i$ is a very ample integral divisor so that, if $c_i= \frac{a_i}{n}$, $D= \sum c_i H_i$.

\begin{itemize}
\item[Ample $\Rightarrow$ (i)]  With the notation of Note \ref{nota:utile}, for every $j \in \{1, \dots, s\}$ there exists $n_{j}$ such that $\mathscr{F}(T_{k_j})(n H_1)$ is globally generated for every $n \geq n_{j}$. Let $n_i'\geq \frac{n_i}{c_1}$ and consider $$m_1= \max_{i}(n_i'),$$
then for every $m \geq m_1$, $$\mathscr{F}([mD])= \mathscr{F}(T_{k_j} + [m c_1]H_1)\otimes \mathcal{O}_X\left(\sum_{i \geq 2}[mc_i]H_i\right) $$ for  any $j$. But this is a tensor product of a globally generated coherent sheaf and a very ample divisor, hence globally generated.

\item[(i) $\Rightarrow$ Ample] We first prove that $[mD]$ is very ample for all $m\geq m_0$. Let $H$ be a very ample divisor and consider $\mathscr{F}=\mathcal{O}_X(-H)$. Then, by (i), there exists $m_0$ such that for all $m \geq m_0$, $[mD]- H$ is globally generated, so that $$[mD]=([mD] - H) + H$$ is a sum of a very ample and a globally generated integral divisor, that is very ample and we are done. Also by Claim \ref{claim:3nef} $D$ is nef.

In particular, if $D$ it is not ample, by Proposition $\ref{pro:ar}$, there exists $0 \ne \gamma \in \overline{\text{NE}}(X)$ such that $(D.\gamma)=0$. If $D=\sum a_i D_i$ $a_i \in \mathbb{R}$, $D_i$ prime divisor, then $$|(\{mD\}.\gamma)| \leq \sum\{ma_i\}|(D_i.\gamma)|< +\infty.$$
Since for $m \geq m_0$ $[mD]$ is very ample, we have that $([mD].\gamma)>0$, and by assumption we obtain $\{mD\}.\gamma<0$. In particular there exists $D_i$ such that $(D_i.\gamma)<0$. Now, for $m \geq m_0$,  $$|(\{mD\}.\gamma)|=- (\{mD\}.\gamma)=([mD].\gamma),$$
and this would imply that $$\lim_{m\to \infty}|(\{mD\}.\gamma)|=+ \infty$$ that is absurd. 
 Also, if we fix any real number $M$, there exists $k \in \mathbb{N}$ such that $-k(D_i.\gamma)> M$.\\
By (i) there exists $m_1$ such that $kD_i + [mD]$ is globally generated for all $m \geq m_1$.

Let us now consider $m \geq m_0,\; m_1$:
$$|(\{mD\}.\gamma)|=([mD].\gamma)=\left(([mD] +kD_i - kD_i).\gamma\right)>M$$
since $[mD] + kD_i$ is nef, and we are done.
\end{itemize}

To conclude we get the statement by Claim \ref{claim:6-10}.
\end{proof}

\begin{re}
The equivalences (ii)-(vi) with the concept of ampleness where already known, the equivalence of (i) with the concept of ampleness is original.
\end{re}

We are finally able to state the results we could achieve for $\mathbb{R}$-divisors.

\begin{pro}[Properties of ampleness for $\mathbb{R}$-divisors]
Let $D\in \text{Div}_{\mathbb{R}}(X)$ be an ample Cartier divisor on a normal projective variety $X$, and let $\mathcal{O}_X([D])$ be the associated line bundle (sometimes we will think $[D]$ as a Weil divisor by the canonical correspondence). Then:

\begin{enumerate}

\item[a)] Given any coherent sheaf $\mathscr{F}$ on $X$, there exists a positive integer $m_1=m_1(\mathscr{F})$ having the property that $$H^i(X, \mathscr{F}\otimes \mathcal{O}_X([mD]))=0 \quad \forall i >0, \; m \geq m_1;$$

%\item[ii)] Given any coherent sheaf $\mathscr{F}$ on $X$, there exists a positive integer $m_2=m_2(\mathscr{F})$ such that $\mathscr{F}\otimes \mathcal{O}_X([mD])$ it is globally generated $\forall m \geq m_2$;

\item[b)] There is a positive integer $m_2$ such that $\mathcal{O}_X([mD])$ is very ample $\forall m \geq m_2$;
 
\item[c)] For every subvariety $V \subseteq X$ of positive dimension, there is a positive integer $m_3=m_3(V)$, such that for every $m \geq m_3$ there exists a non-zero section $0\ne s =s_{V,m} \in H^0(V, \mathcal{O}_V([mD]))$, such that $s$ vanishes at some point of $V$;

\item[d)] For every subvariety $V \subseteq X$ of positive dimension, \\$\chi(V, \mathcal{O}_V([mD]))\to +\infty$ as $m \to +\infty$;

\end{enumerate}
\end{pro}
\begin{proof}
If $D$ is ample then $D=\sum a_i A_i$ where $A_i$ is an ample integral divisor and $a_i>0$, $a_i\in \mathbb{R}$. We can now consider $n \gg0$ such that $nA_i=H_i$ is a very ample integral divisor so that, if $c_i= \frac{a_i}{n}$, $D= \sum c_i H_i$.
\begin{itemize}

\item[Ample $\Rightarrow$ (a)] In  analogous way to Ample $\Rightarrow$ (i) of  Proposition \ref{pro:air} and by Fujita's Theorem, there exists an integer $m_2$ such that  $$H^i(X,\mathscr{F}([mD]))=  H^i \left(X,\mathscr{F}(T_{k_j} + [m c_1]H_1)\otimes \mathcal{O}_X\left(\sum_{s \geq 2}[mc_s]H_s\right)\right) =0$$ for all $m\geq m_2$ and all $i \geq 1$.

\item[Ample $\Rightarrow$ (b)] In  analogous way to Ample $\Rightarrow$ (i) of  Proposition \ref{pro:air} with $\mathscr{F}= \mathcal{O}_X$.
%In fact we get that if $D$ is an ample $\mathbb{R}$-divisor, it is enough to consider the equality $$mD = [mD] + \{ mD \} \Rightarrow D = \frac{[mD]}{m} + \frac{\{mD\}}{m}.$$
%If we now consider $m \to \infty$ we get that $\frac{\{mD\}}{m} \stackrel{m \to \infty}{\longrightarrow} 0  \Rightarrow \frac{[mD]}{m} \stackrel{m \to \infty}{\longrightarrow} D$. For the openness of the ample cone we get that there exists a neighborood of $D \in U \subseteq \text{Amp}(X)$ and an integer $m_0 \gg 0$ such that for every $m \geq m_0 \quad \frac{[mD]}{m} \in U$ and so $m \frac{[mD]}{m}= [mD]$ is ample.

\item[Ample $\Rightarrow$ (c)] Passing through the property (b) we get the statement. In fact $(c)$ obviously holds for any very ample divisor.

\item[Ample $\Rightarrow$ (d)] Let us consider two ample integral effective divisors $A, B \subseteq V$ and the canonical exact sequence of $A$:
$$0 \to \mathcal{O}_V(B) \to  \mathcal{O}_V(A+B) \to \mathcal{O}_A( A + B) \to 0 $$
so that we obtain the cohomological long exact sequence: 
$$0 \to H^0(\mathcal{O}_V(B)) \stackrel{f}{\rightarrow}  H^0(\mathcal{O}_V(A+B)) \stackrel{h}{\rightarrow} H^0(\mathcal{O}_A( A + B))\to \cdots$$  
By the ampleness of $A$ and $B$, we get that $f$ cannot be surjective, so that $h$ is not the zero map, whence
$$h^0(\mathcal{O}_V( A + B))- h^0(\mathcal{O}_V(B)) \geq 1.$$

By \emph{ample $\Rightarrow$ (b)} we know that there exists a positive integer $m_0$, such that $[mD]$ is very ample for all $m \geq m_0$.
Also, by Note \ref{nota:utile},  $[mD]= \sum [m c_i]H_i + T_m$.\\
Let $n> m>m_0$, such that $T_m= T_n$; then we obtain a new divisor in the form $[nD]= \sum [n c_i]H_i + T_{n}$, that is $[nD]= [mD] + \sum e_i H_i$, $e_i \in \mathbb{N}$ where, if $e_i=0\; \forall i$, we take $n' >n$. Let us consider an increasing sequence of those $n$.\\
Since $[mD]$ and $ \sum e_i H_i$ are very ample, we get that $$h^0(\mathcal{O}_V([nD]|_V)) \stackrel{n \to \infty}{\longrightarrow} + \infty.$$

To conclude, since \emph{ample $\Rightarrow$ (a)} we get $h^i(V, \mathcal{O}_V([mD]|_V))=0$ for all $i \geq 1$ and for all $m\gg0$, whence: $$\chi(\mathcal{O}_V([nD])) \stackrel{n \to \infty}{\longrightarrow} + \infty.$$

\end{itemize}
\end{proof}

\section{Big line bundles and divisors}

\begin{de}[Big] \label{de:big}
A line bundle $\mathscr{L}$ on a projective variety $X$ is \emph{big} if $\kappa(X, \mathscr{L})=\dim X$. A Cartier divisor $D$ on $X$ is \emph{big} if $\mathcal{O}_X(D)$ is so.
\end{de}

Also reacall that
\begin{de}
The semigroup of divisor $D$ on a projective variety $X$ is defined as $$\mathbf{N}(X, D) = \{m \in \mathbb{N}\,| \, H^0(X, \mathcal{O}_X(mD)) \ne 0\}.$$
\end{de}

\begin{lem}\label{lem:b1}
Assume that $X$ is a projective variety of dimension $n$. A divisor $D$ on $X$ is big if and only if there is a constant $C>0$ such that $$h^0(X, \mathcal{O}_X(mD))\geq C \cdot m^n$$ for all sufficiently large $m \in \mathbf{N}(X, D)$.
\end{lem}

\begin{pro}[Kodaira's lemma] Let $D$ be a big Cartier divisor and $F$ an arbitrary effective Cartier divisor on $X$. Then $$H^0(X, \mathcal{O}_X(mD- F)) \ne 0$$
for all sufficiently large $m \in \mathbf{N}(X, D)$.
\end{pro}

\begin{proof}
Suppose that $\dim X= n$ and consider the exact sequence of $F$ $$0\to \mathcal{O}_X(-F) \to \mathcal{O}_X \to \mathcal{O}_F \to 0$$
tensored by $\mathcal{O}_X(mD)$:
$$0\to \mathcal{O}_X(mD-F) \to \mathcal{O}_X(mD) \to \mathcal{O}_F(mD) \to 0.$$
Since $D$ is big, by the Lemma \ref{lem:b1} there is a constant $C>0$ such that $h^0(X, \mathcal{O}_X(mD))\geq c\cdot m^n$ for sufficiently large $m \in \mathbf{N}(X, D)$. On the other hand $\dim F= n-1$ so that $h^0(F,\mathcal{O}_F(mD))$ grows at most like $m^{n-1}$. Therefore $$h^0(X, \mathcal{O}_X(mD))> h^0(F, \mathcal{O}_F(mD)$$ for large $m\in \mathbf{N}(X, D)$ and the assertion follows by the exact sequence.
\end{proof}

\begin{co}[Characterization of big integral divisors]\label{co:big} Let $D$ be a Cartier divisor on a normal variety $X$. Then the following are equivalent:
\begin{enumerate}

\item $D$ is big;

\item there exists an integer $a \in \mathbb{N}$ such that $\varphi_{|mD|}$ is birational for all $m \in \mathbf{N}(X, D)_{\geq a}$;
 
\item $\varphi_{|mD|}$ is generically finite for some $m \in \mathbf{N}(X, D)$;

\item for any coherent sheaf $\mathscr{F}$ on $X$, there exists a positive integer $m=m(\mathscr{F})$ such that $\mathscr{F}\otimes \mathcal{O}_X(mD)$ is generically globally generated, i.e. is such that the natural map $$H^0(X, \mathscr{F} \otimes\mathcal{O}_X(mD)) \otimes_{\mathbb{C}}\mathcal{O}_X \to \mathscr{F} \otimes \mathcal{O}_X(mD)$$ is surjective over a dense open subset;

\item for any ample integral divisor $A$ on $X$, there exists a positive integer $m>0$, and an effective divisor $N$ such that $mD\equiv_{lin}A +N$;

\item same as in (5) for some integral ample divisor $A$;

\item there exists an ample integral divisor $A$, a positive integer $m >0$ and an effective divisor $N$ such that $mD\equiv_{num} A+N$.

\end{enumerate}

\end{co}

\begin{de}[Big $\mathbb{Q}$-divisors] \label{de:bigr}
A $\mathbb{Q}$-divisor $D$ is big if there is a positive integer $m>0$ such that $mD$ is integral and big.
\end{de}

\begin{de}[Big $\mathbb{R}$-divisors]\label{de:bigir}
An $\mathbb{R}$-divisor $D \in \text{Div}_{\mathbb{R}}(X)$ is big if it can written in the form $$D= \sum a_i \cdot D_i$$ where each $D_i$ is a big integral divisor and $a_i$ is a positive real number.  
\end{de}

\begin{pro}\label{pro:neb} Let $D$ and $D'$ be $\mathbb{R}$-divisors on $X$. If $D \equiv_{num} D'$, then $D$ is big if and only if $D'$ is big.
\end{pro}

\begin{re}If $D$ is an integral divisor, $D$ is big in the sense of $\mathbb{Z}$-divisors (\ref{co:big}) if and only if it is big in the sense of $\mathbb{R}$-divisors (\ref{de:bigir}).
\end{re}

\begin{proof}If $D$ is big in the sense of $\mathbb{Z}$-divisors, obviously $D$ can be written as $1 \cdot D$, so that it is a big $\mathbb{R}$-divisor.\\
If $D=\sum c_i B_i $, where $c_i \in \mathbb{R}$, $c_i>0$ and $B_i$ is a big integral divisor, as we will see in the Claim \ref{claim:boh}, there exists $m_0$ such that $[m_0D]$ is an integral big divisor. But in this case we have that $[m_0D]=m_0D$ and we get the statement by Corollary \ref{co:big}. 

\end{proof}

As we have just done for the ampleness, we would extend the various properties of bigness, when it is possible, to $\mathbb{Q}$ and $\mathbb{R}$-divisors referring us to Corollary \ref{co:big}; \\
The first step will be to redefine the notion of semigroup:

\begin{de} Let $D$ be a $\mathbb{R}$-divisor; the \emph{semigroup} of $D$ is the set $$\mathbf{N}(D)= \mathbf{N}(X, D)= \{m \geq 0 | H^0(X, \mathcal{O}_X([mD])) \ne 0\}.$$
\end{de}

\begin{pro}[Bigness for $\mathbb{Q}$-divisors]\label{pro:bigr}  Let $D$ be a $\mathbb{Q}$-divisor on a projective variety $X$. Then the following are equivalent:
\begin{itemize}

\item[I)] $D$ is big;

\item[II)] there exists an integer $a \in \mathbb{N}$ such that $\varphi_{|[mD]|}$ is birational for all $m \in \mathbf{N}(X, D)_{\geq a}$;
 
\item[III)] $\varphi_{|[mD]|}$ is generically finite for some $m \in \mathbf{N}(X, D)$;

\item[IV)] for any coherent sheaf $\mathscr{F}$ on $X$, there exists a positive integer $m=m(\mathscr{F})$ such that $\mathscr{F}\otimes \mathcal{O}_X([mD])$ is generically globally generated, that is such that the natural map $$H^0(X, \mathscr{F} \otimes\mathcal{O}_X([mD])) \otimes_{\mathbb{C}}\mathcal{O}_X \to \mathscr{F} \otimes \mathcal{O}_X([mD])$$ is generically surjective;

\item[V)] for any ample $\mathbb{Q}$-divisor $A$ on $X$, there exists an effective $\mathbb{Q}$-divisor $N$ such that $D\equiv_{lin}A +N$;

\item[VI)] same as in (V) for some ample $\mathbb{Q}$-divisor $A$;

\item[VII)] there exists an ample $\mathbb{Q}$-divisor $A$ on $X$ and an effective $\mathbb{Q}$-divisor $N$ such that $D\equiv_{num} A+N$. 

\end{itemize}

\end{pro}

\begin{proof} \emph{}

\begin{itemize}
\item[(I)$\Rightarrow$ (II)] If $D$ is a big $\mathbb{Q}$-divisor, we know that there exists $k \gg 0$ such that $kD$ is a multiple of an integral big divisor, so that, by Corollary \ref{co:big}, $kD\equiv_{lin} A + E $, $A$ ample and $E$ effective $\mathbb{Z}$-divisors. By Lemma  \ref{lem:dr}: $$[mD]\equiv_{lin} tkD + [iD] \equiv_{lin} tA + t E +[iD], \qquad 0\leq i \leq k-1.$$

Let $r \in \mathbb{N}$ such that $rA$ is very ample and let $s \in \mathbb{N}$ such that $sA + [iD]$ is globally generated for all $i=1, \dots, k-1$. Then, for every $m \geq k(r + s)$, we obtain:
$$[mD]\equiv_{lin} (t-s)A + (sA + [iD]) + tE= H + tE ,$$ where $H$ is very ample (for very ample + globally generated is very ample) and $tE$ is effective.
To conclude $$\kappa(X, [mD]) \geq \kappa(X, H) = \dim X,$$
hence $\kappa(X, [mD])= \dim X$.
\item[(II)$\Rightarrow$ (I)] If $(II)$ holds it is also true that $\varphi_{|m(kD)|}$ is birational for all $m \in \mathbf{N}(X, kD)_{\geq ka}$. Now, by Corollary \ref{co:big}, $kD$ is a big integral divisor, and so $D$ is a big $\mathbb{Q}$-divisor by definition.

\item[(III),...,(VII) $\Leftrightarrow$ (I)]
The implication \emph{(I) $\Rightarrow$ (III),...,(VII)} is obvious, in fact if $D$ is big we only need to consider an integer $m$ for which $mD=D'$ is an integral big divisor. Now for those properties it is sufficient that there exists an integer satisfying them. Then we only need to consider the product of this integer with $m$ and we obtain the statement.\\
Now we need to prove that if those properties hold, $D$ is a big divisor.
\begin{itemize}
\item[(III) $\Rightarrow$ (I)] By (III) and by Corollary \ref{co:big} we know that $[mD]$ is big; now $mD= [mD] + \{mD\}$ where $\{mD\}$ is an effective $\mathbb{Q}$-divisor. Also, there exists an integer $k>0$ such that $kmD$ is an integral divisor, where $$kmD= k[mD]+ k\{mD\}$$ that is a sum of a big and an effective integral divisor. Now, by Corollary \ref{co:big}, $kmD$ is big and accordingly $D= \frac{1}{km}(kmD)$ is.

\item[(IV) $\Rightarrow$ (V)] For every ample $\mathbb{Q}$-divisor A, let us consider an effective integral divisor $E$ such that $E -\{A\}$ is effective.  Let us consider $\mathscr{F}=\mathcal{O}_X(-[A] -E)$. Then, by (IV)  there exists $m$ such that $\mathcal{O}_X(-[A]- E) \otimes \mathcal{O}_X([mD])$ is generically globally generated. This implies $H^0(X, \mathcal{O}_X([mD]-[A]-E)) \ne 0$, so that there exists an effective integral divisor $F$ giving
\begin{align*}[mD]\equiv_{lin} [A] + E +F \Rightarrow mD &\equiv_{lin} [mD]+ \{mD\} \equiv_{lin}\\ &\equiv_{lin} A + (E-\{A\})+(F+ \{mD\}),\end{align*} as we wanted.

\item[(V) $\Rightarrow$ (I)] if this property holds, there exists an integer $k>0$ such that $kmD$, $kA$ and $kN$ are integral divisors, where $kA$ is ample and $kN$ is effective. Now, by Corollary \ref{co:big}, $kmD$ is big, and so $D= \frac{1}{km}(kmD)$ is.

\item[(VI) $\Rightarrow$ (I)] like in $(V)$.

\item[(VII) $\Rightarrow$ (I)] like in $(V)$.

\end{itemize}
%\end{claim}
\end{itemize}

\end{proof}

\begin{re} \label{re:bqr}If $B$ is a big rational divisor and $N$ is an effective rational divisor, then $B+sN$ is big for all $s\in \mathbb{R}, s>0$.
\end{re}
\begin{proof}
If $s \in \mathbb{Q}$ it is obvious by Proposition $\ref{pro:bigr}$. If $s \in \mathbb{R}- \mathbb{Q}$ we only need to choose two positive rational numbers $s_1, s_2$ with $s_1 < s < s_2$ and $t \in [0,1]$ such that $s=ts_1 +(1 - ts_2)$. Then $$B + s N = t(B + s_1 N)+ (1- t)(B + s_2 N)$$ that is a positive linear combination of big $\mathbb{Q}$-divisors.
\end{proof}

\begin{claim}\label{claim:boh} Let $D$ be an $\mathbb{R}$-divisor on a projective variety $X$. Then the following are equivalent:
\begin{enumerate}
\item $D$ is big;
\item There exists an integer $m_0>0$ such that $[mD]$ is a big integral divisor for all $m \geq m_0$;
\item There exists an integer $m_1>0$ such that $[m_1D]$ is a big integral divisor.
\end{enumerate}
\end{claim}
\begin{proof} \emph{}

\begin{itemize}
\item[1 $\Rightarrow$ 2)] If $D$ is a big divisor then $D=\sum a_i D_i$, $a_i \in \mathbb{R}$, $a_i >0$ and $D_i \in \text{Div}(X)$ big divisors. By Note \ref{nota:utile} we can write $[mD]= \sum[ma_i]D_i + T_k$ for finitely many integral divisors $T_k$. We also point out that $\{mD\} \in \text{Div}_{\mathbb{R}}(X)$ is an effective divisor.\\
By Corollary \ref{co:big} 
  $$ [mD]\equiv_{lin} \sum [ma_i](A_i+ E_i) + T_k $$ where $A_i$ is an ample integral divisor and $E_i$ is an effective integral divisor. Now we can choose $r	\in \mathbb{N}$ such that $tA_i$ is very ample for every $i$ and for every $t\geq r$. Let $s \in \mathbb{N}$ such that $sA_1 + T_k$ is globally generated for every $k$. If we take $m_0$ such that $[m_0a_i] \geq r \; \forall i$ and $[m_0a_1] \geq r + s$, then for all $m \geq m_0$
$$[mD] \equiv_{lin} \sum_{i \geq 2} [ma_i](A_i + E_i) + ([ma_1] - s)A_1 + [ma_1]E_1 +(sA_1 +T_k)$$ where $\sum_{i \geq 2} [ma_i]A_i + ([ma_1] - s)A_1 +(sA_1 +T_k) = H$ is a very ample integral divisor and $\sum_i [ma_i]E_i = E$ is an effective integral divisor so that $$\kappa(X, [mD]) \geq \kappa(X, H) = \dim X,$$ and we get the statement.

\item[2 $\Rightarrow$ 3)] Trivial.

\item[3 $\Rightarrow$ 1)]We have that $[m_1D]$ is big.
Now by Proposition \ref{co:big} there exist an ample integral divisor $A$ and an effective integral divisor $E$ such that $$[m_1D] \equiv_{num} A + E \Rightarrow m_1D \equiv_{num} A + (E + \{ m_1D \})$$ that is the sum of an ample and an effective $\mathbb{R}$-divisor. By Proposition \ref{pro:neb} it is enough to prove that $A +E + \{mD\}$ is big. So we get the statement by Remark \ref{re:bqr}: if $D= \sum a_i D_i$, $a_i\in\mathbb{R}$, $D_i$ prime divisors, then $$\{m_1D\}= \sum_{i=1}^{s}\{m_1a_i\}D_i$$ and we can write $$A + (E + \{ m_1D \})=\sum_{i=1}^{s}\frac{1}{s}(A+ E + s\{m_1a_i\}D_i) .$$
\end{itemize}
\end{proof}

\vskip 0.5cm
Now we are discussing the case of $\mathbb{R}$-divisors:

\begin{pro}[Bigness for $\mathbb{R}$-divisors]\label{pro:bigir} Let $D$ be an $\mathbb{R}$-divisor on a projective variety $X$. The following are equivalent:
\begin{enumerate}

\item[(i)] $D$ is big;

\item[(ii)] there exists an integer $a \in \mathbb{N}$ such that $\varphi_{|[mD]|}$ is birational for all $m \in \mathbf{N}(X, D)_{\geq a}$;
 
\item[(iii)] $\varphi_{|[mD]|}$ is generically finite for some $m \in \mathbf{N}(X, D)$;

\item[(iv)] for any coherent sheaf $\mathscr{F}$ on $X$, there exists a positive integer $m=m(\mathscr{F})$ such that $\mathscr{F}\otimes \mathcal{O}_X([mD])$ is generically globally generated, that is such that the natural map $$H^0(X, \mathscr{F} \otimes\mathcal{O}_X([mD])) \otimes_{\mathbb{C}}\mathcal{O}_X \to \mathscr{F} \otimes \mathcal{O}_X([mD])$$ is generically surjective;

\item[(v)] for any ample $\mathbb{R}$-divisor $A$ on $X$, there exists an effective $\mathbb{R}$-divisor $N$ such that $D \equiv_{num}A +N$;

\item[(vi)] same as in (v) for some ample $\mathbb{R}$-divisor $A$;

\end{enumerate}
\end{pro}

\begin{proof}

\emph{}

\begin{itemize}
\item[(i)$\Rightarrow$ (ii) ] As in the proof of Claim \ref{claim:boh} there exists a positive integer $m_0$ such that we can write $$[mD]\equiv_{lin}H + E$$ for all $m \geq m_0$, where $H$ is a very ample integral divisor and $E$ is an effective integral divisor. Then, obviously, $\varphi_{|[mD]|}$ is birational for all $m\geq m_0$.

\item[(ii)$\Rightarrow$ (iii) ] Trivial.

\end{itemize}

The implication \emph{(iii)$\Rightarrow$ (i)} is obvious, in fact like above we only need to consider the integer $m$ for which $[mD]$ is big and the statement follows by Claim \ref{claim:boh}.

\vskip 0.5cm

\begin{itemize}

\item[(i)$\Rightarrow$ (iv) ]  $D$ is a big $\mathbb{R}$-divisor, then $D=\sum a_i D_i$, $a_i \in \mathbb{R}$, $a_i >0$ and $D_i$ big integral divisor. Also by Note \ref{nota:utile} we can write $[mD] = \sum [ma_i]D_i + T_k$ for finitely many integral divisors $T_k$. Let $A$ be an ample integral divisor, then there exists $m_0=m(\mathscr{F}(T_k))$ such that $\mathscr{F}(T_k)(m_0A)$ is globally generated for all $k$. 
Let us denote $m_0A=H$\\
Since $D_i$ is a big integral divisor, by Corollary \ref{co:big} there exists $m_i \in \mathbb{N}$ such that $m_iD_i \equiv_{lin} H + E_i$ where $E_i$ is an effective integral divisor. Also, since $D_i$ is big and integral, there exists $n_i$ such that $n D_i $ is effective for all $n \geq n_i$.  

Let  $m\gg 0 $ such that $[ma_i] - m_i\geq n_i$  for all $i$, then 
$$\mathscr{F}(T_k)([mD])= \mathscr{F}(T_k)(\sum_{i=1}^s (([ma_i] - m_i)D_i + m_iD_i))$$
where $([ma_i] - m_i)D_i$ is effective and $m_iD_i \equiv_{lin} H + E_i$.\\ 
Then $\mathscr{F}(T_k)([mD])=\mathscr{F}(T_k)(sH + E)$ where $E$ is effective and  $\mathscr{F}(T_k)(sH)$ is globally generated and we are done.

\item[(iv)$\Rightarrow$ (v) ] Like in Proposition \ref{pro:bigr} \emph{(IV)$\Rightarrow$ (V)} where we replace $\equiv_{lin}$ by $\equiv_{num}$.

\item[(v)$\Rightarrow$ (vi) ] Trivial. 

\item[(vi)$\Rightarrow$ (i) ] By (vi) we have that $D\equiv_{num}A + N$ where $A$ is an ample and $N$ is an effective $\mathbb{R}$-divisor. By the openness of amplitude,there exist $E_1, \ldots, E_r$ effective $\mathbb{R}$-divisors, such that the divisor $$H = A -\varepsilon_1 E_1 - \ldots -\varepsilon_rE_r $$ is an ample $\mathbb{Q}$-divisor, for $0<\varepsilon_i \ll 1$. Also, $M = N + \varepsilon_1 E_1 + \ldots +\varepsilon_rE_r$ is an effective $\mathbb{R}$-divisor.\\ 
Now we can write $M= \sum_{i=1}^s c_i D_i$ where $c_i\in \mathbb{R}, \; c_i>0$ and $D_i$ is a prime divisor. \\
Then the results follows by Remark \ref{re:bqr} and Proposition \ref{pro:neb}, since we can write $$D \equiv_{num} \sum_{i=1}^{s} \left( \frac{1}{s}H + c_iD_i\right).$$

\end{itemize}

\end{proof}

\addcontentsline{toc}{section}{Bibliography}
\nocite{*}
\bibliographystyle{alpha}  %altri stili: alpha, plain, unsrt, abbrv
\bibliography{bibliografia}

\begin{thebibliography}{{Rei}97}

\bibitem[Bea96]{MR1406314}
Arnaud Beauville.
\newblock {\em Complex algebraic surfaces}, volume~34 of {\em London
  Mathematical Society Student Texts}.
\newblock Cambridge University Press, Cambridge, second edition, 1996.
\newblock Translated from the 1978 French original by R. Barlow, with
  assistance from N. I. Shepherd-Barron and M. Reid.

\bibitem[CCP05]{campana-2005}
Frederic Campana, Jungkai~A. Chen, and Thomas Peternell.
\newblock Strictly nef divisors, 2005.

\bibitem[CP90]{MR1048532}
Fr{\'e}d{\'e}ric Campana and Thomas Peternell.
\newblock Algebraicity of the ample cone of projective varieties.
\newblock {\em J. Reine Angew. Math.}, 407:160--166, 1990.

\bibitem[Die72]{MR0308117}
J.~Dieudonn{\'e}.
\newblock The historical development of algebraic geometry.
\newblock {\em Amer. Math. Monthly}, 79:827--866, 1972.

\bibitem[Har70]{MR0282977}
Robin Hartshorne.
\newblock {\em Ample subvarieties of algebraic varieties}.
\newblock Notes written in collaboration with C. Musili. Lecture Notes in
  Mathematics, Vol. 156. Springer-Verlag, Berlin, 1970.

\bibitem[Har77]{MR0463157}
Robin Hartshorne.
\newblock {\em Algebraic geometry}.
\newblock Springer-Verlag, New York, 1977.
\newblock Graduate Texts in Mathematics, No. 52.

\bibitem[Ken83]{MR691366}
Keith~M. Kendig.
\newblock Algebra, geometry, and algebraic geometry: some interconnections.
\newblock {\em Amer. Math. Monthly}, 90(3):161--174, 1983.

\bibitem[Kle66]{MR0206009}
Steven~L. Kleiman.
\newblock Toward a numerical theory of ampleness.
\newblock {\em Ann. of Math. (2)}, 84:293--344, 1966.

\bibitem[KN74]{MR55555}
L.~Kuipers and H.~Niederreiter.
\newblock {\em Uniform Distribution of Sequences}.
\newblock Pure and Applied Mathematics. New York-London-Sydney:
  Wiley-Interscience (John Wiley and Sons), New York, 1974.

\bibitem[Kol96]{MR1440180}
J{\'a}nos Koll{\'a}r.
\newblock {\em Rational curves on algebraic varieties}, volume~32 of {\em
  Ergebnisse der Mathematik und ihrer Grenzgebiete. 3. Folge. A Series of
  Modern Surveys in Mathematics [Results in Mathematics and Related Areas. 3rd
  Series. A Series of Modern Surveys in Mathematics]}.
\newblock Springer-Verlag, Berlin, 1996.

\bibitem[Laz04a]{MR2095471}
Robert Lazarsfeld.
\newblock {\em Positivity in algebraic geometry. {I}}, volume~48 of {\em
  Ergebnisse der Mathematik und ihrer Grenzgebiete. 3. Folge. A Series of
  Modern Surveys in Mathematics [Results in Mathematics and Related Areas. 3rd
  Series. A Series of Modern Surveys in Mathematics]}.
\newblock Springer-Verlag, Berlin, 2004.
\newblock Classical setting: line bundles and linear series.

\bibitem[Laz04b]{MR2095472}
Robert Lazarsfeld.
\newblock {\em Positivity in algebraic geometry. {II}}, volume~49 of {\em
  Ergebnisse der Mathematik und ihrer Grenzgebiete. 3. Folge. A Series of
  Modern Surveys in Mathematics [Results in Mathematics and Related Areas. 3rd
  Series. A Series of Modern Surveys in Mathematics]}.
\newblock Springer-Verlag, Berlin, 2004.
\newblock Positivity for vector bundles, and multiplier ideals.

\bibitem[Mor87]{MR927961}
Shigefumi Mori.
\newblock Classification of higher-dimensional varieties.
\newblock In {\em Algebraic geometry, Bowdoin, 1985 (Brunswick, Maine, 1985)},
  volume~46 of {\em Proc. Sympos. Pure Math.}, pages 269--331. Amer. Math.
  Soc., Providence, RI, 1987.

\bibitem[{Rei}97]{reid-1996}
Miles {Reid}.
\newblock {Chapters on algebraic surfaces.}
\newblock In {\em {Complex algebraic geometry. Lectures of a summer program,
  Park City, UT, 1993}}, pages 5--159. Providence, RI: American Mathematical
  Society, 1997.

\bibitem[Ser95]{MR1340341}
Fernando Serrano.
\newblock Strictly nef divisors and {F}ano threefolds.
\newblock {\em J. Reine Angew. Math.}, 464:187--206, 1995.

\bibitem[Sha94]{MR1328834}
Igor~R. Shafarevich.
\newblock {\em Basic algebraic geometry. 2}.
\newblock Springer-Verlag, Berlin, second edition, 1994.
\newblock Schemes and complex manifolds, Translated from the 1988 Russian
  edition by Miles Reid.

\end{thebibliography}

\end{document}